\documentclass[dvipdfmx]{article}

\usepackage{amssymb,amsmath,amsthm}
\usepackage{hyperref}
\usepackage{authblk}
\usepackage{tikz}
\usepackage{enumitem}
\usepackage{ascmac}

\theoremstyle{definition}
\newtheorem{theorem}{Theorem}[section]
\newtheorem{remark}[theorem]{Remark}
\newtheorem{lemma}[theorem]{Lemma}
\newtheorem{definition}[theorem]{Definition}
\newtheorem{corollary}[theorem]{Corollary}

\newtheorem{proposition}[theorem]{Proposition}

\newcommand{\LPP}{\mathsf{LPP}}
\newcommand{\RCAo}{\mathsf{RCA}_0}
\newcommand{\ACAo}{\mathsf{ACA}_0}
\newcommand{\HJ}{\mathrm{HJ}}
\newcommand{\hyp}{\mathchar`-}
\newcommand{\PCAo}{\Pi^1_1\hyp\mathsf{CA}_0}
\newcommand{\ATRo}{\mathsf{ATR}_0}
\newcommand{\TI}{\mathsf{TI}}
\newcommand{\TLPP}{\mathsf{TLPP}}
\newcommand{\ALPP}{\mathsf{ALPP}}
\newcommand{\RFN}{\mathsf{RFN}}
\newcommand{\DC}{\mathsf{DC}}
\newcommand{\M}{\mathcal{M}}
\newcommand{\Ef}{\mathrm{R}}
\newcommand{\N}{\mathbb{N}}
\newcommand{\LP}{\mathrm{LP}}
\DeclareMathOperator{\TJ}{\mathrm{TJ}}
\newcommand{\T}{\mathrm{T}}
\newcommand{\forces}{\Vdash}
\newcommand{\WO}{\mathrm{WO}}
\newcommand{\NN}{\mathcal{N}}
\newcommand{\ACA}{\mathsf{ACA}}
\newcommand{\IND}{\mathsf{IND}}
\newcommand{\KB}{\mathrm{KB}}
\newcommand{\GN}[1]{\ulcorner #1 \urcorner}

\renewcommand{\P}{\mathsf{P}}
\newcommand{\Q}{\mathsf{Q}}
\newcommand{\W}{\mathrm{W}}
\newcommand{\dom}{\mathrm{dom}}

\title{Relative leftmost path principles and omega-model reflections of transfinite inductions}

\author[1]{Yudai Suzuki}
\affil[1]{yudai.suzuki.research@gmail.com}
\date{\today}
\begin{document}
\maketitle

\begin{abstract}
In this paper, we give characterizations of Towsner's relative leftmost path principles in terms of omega-model reflections of transfinite inductions. In particular, we show that the omega-model reflection of $\Pi^1_{n+1}$ transfinite induction is equivalent to the $\Sigma^0_n$ relative leftmost path principle over $\RCAo$ for $n > 1$. As a consequence, we have that $\Sigma^0_{n+1}\LPP$  is strictly stronger than $\Sigma^0_{n}\LPP$.
\end{abstract}

\section{Introduction}
In recent studies of reverse mathematics, not only the \textit{big five} systems, but also other systems are focused on.
For example, it is known that the strength of Kruskal's tree theorem is precisely between $\ATRo$ and $\PCAo$, and the best-known upper bound for Menger's theorem in graph theory is also between $\ATRo$ and $\PCAo$.

For studying intermediate $\Pi^1_2$ statements between $\ATRo$ and $\PCAo$, 
Towsner introduced new $\Pi^1_2$ principles which we call pseudo leftmost path principles\footnote{Originally they were called relative leftmost path principles.} stating that any ill-founded tree has a \textit{pseudo} leftmost path \cite{Towsner_TLPP}. Here, a pseudo leftmost path is a path which may not be the actual leftmost path, but it behaves like a leftmost path in a restricted range. 
In \cite{suzuki_yokoyama_pi12}, the author and Yokoyama generalized Towsner's work and gave some characterizations of $\{\sigma \in \Pi^1_2: \PCAo \vdash \sigma\}$. They introduced \textit{pseudo} hyperjumps and compared them with pseudo leftmost path principles. 

The works in \cite{suzuki_yokoyama_pi12} and  \cite{Freund} suggest that 
a finer analysis of the relationship between the $n$-th and $n+1$-th pseudo hyperjumps is also interesting.
Since Towsner's $\Sigma^0_n$ leftmost path principles approximate the single pseudo hyperjump, 
it seems helpful to study them for a finer analysis between pseudo hyperjumps.

In this paper, we give new characterizations of pseudo leftmost path principles.
We prove the following equivalences over $\RCAo$:
\begin{itemize}
  \item the transfinite leftmost path principle and the single pseudo hyperjump with the base $\Pi^1_1\hyp\TI$,
  \item the arithmetic leftmost path principle, the single pseudo hyperjump and the omega-model reflection of $\Pi^1_{\infty}\hyp\TI$,
  \item the $n$-th pseudo hyperjump with the base $\Pi^1_{\infty}\hyp\TI$ and  the $n+1$-th pseudo hyperjump,
  \item the $\Sigma^0_n$ leftmost path principle and the omega-model reflection of $\Pi^1_{n+1}\hyp\TI$ for $n > 1$.
\end{itemize}
We note that these results help to reveal the proof-theoretic strength of pseudo leftmost path principles.
We also note that the second clause is independently proved by Freund \cite{Freund}.

\subsection*{Structure of this paper}
In Section 2, we recall the notion of coded $\omega$-models and the $\omega$-model reflections. Then, we introduce the notion of realizers for coded $\omega$-models, which is used to prove results in Section 3.
In Section 3, we prove the main results. 
In Section 4, we rephrase the results in Section 3 from the point of view of Weihrauch reduction.

\subsection*{Acknowledgements}
The author would like to thank Keita Yokoyama, Anton Freund and Koshiro Ichikawa for their useful comments and discussions for this paper.

\section{Coded omega-models}
In this section, we see some properties of coded $\omega$-models.
For the basic definitions, see also Simpson's textbook \cite{Simpson}.

\begin{definition}[$\RCAo$]
  For a set $X$, we define the $i$-th segment $X_i$ by
  $X_i = \{j : \langle i,j \rangle \in X\}$ where $\langle i,j \rangle$ denotes the standard pairing function from $\N^2$ to $\N$. Then, a set $X$ can be regarded as a sequence $\langle X_i \rangle_{i \in \N}$.
  In this sense, we say a set $A$ is contained  in $X$ (written $A \in X$) if $\exists i(A = X_i)$ holds.
\end{definition}

When we identify a set $\M$ with the structure $(\N,\langle \M_i\rangle_{i \in \N})$, we call $\M$ a coded $\omega$-model. If it is clear from the context, we sometimes omit the word \textit{coded}, and simply say $\M$ is an $\omega$-model.
One of the important notions related to coded $\omega$-models is the principle of $\omega$-model reflection. We introduce two kinds of $\omega$-model reflections.
\begin{definition}
  Let $T$ be a recursively axiomatized theory.
  We define the $\omega$-model reflection of $T$ (written $\RFN(T)$) as the assertion that 
  any set $X$ is contained in an $\omega$-model of $T + \ACAo$.
\end{definition}
It is easy to see that $\RFN(T)$ implies the consistency of $\ACAo + T$ over $\RCAo$.
Therefore, sometimes $\RFN(T)$ is used for separating two theories.

\begin{proposition}
Let $T$ and $T'$ be recursive consistent extensions of $\ACAo$ such that 
$T$ proves $\RFN(T')$. Then, $T'$ does not prove $T$. If $T$ proves $T'$ in addition, then $T$ is strictly stronger than $T'$.
\end{proposition}
\begin{proof}
If $T'$ proved $T$, then $T'$ would prove the consistency of $T'$, a contradiction. 
\end{proof}

\begin{definition}
  Let $\Gamma$ be a class of formulas. We define 
  the $\Gamma$-reflection schema (written $\Gamma\hyp\RFN$) as the collection of 
  \begin{align*}
    \forall X(\varphi(X) \to \exists \M \text{: a coded $\omega$-model } (\M \models \ACAo +  \varphi(X)))
  \end{align*}
  for all $\varphi \in \Gamma$.
\end{definition}
The following characterization of $\Pi^1_{n}\hyp\RFN$ is well-known.
\begin{theorem}[$\ACAo$]
\cite{Jager_and_Strahm}
  Let $n \in \omega$ such that $n >0$. 
  Then, $\Pi^1_{n+1}\hyp\RFN$ is equivalent to $\Pi^1_{n}\hyp\TI$.
\end{theorem}

In the usual definition in the context of formal logics, the truth of a sentence $\sigma$ over a structure $\M$ is defined by Tarski's truth definition. However,  it is practically identified with the truth of the relativization $\sigma^{\M}$.
We note that when working in a weak system like $\RCAo$, we need to distinguish between these two kinds of truth: $\sigma^{\M}$ and $\M \models \sigma$.
For a detailed definition of $\M \models \sigma$, see \cite{Simpson}.

\begin{definition}[$\RCAo$]
  Let $\M = \langle \M_i \rangle_{i \in \N}$ be a sequence of sets. 
  We say $\M$ is a jump ideal if 
  \begin{align*}
    \forall i (\TJ(\M_i) \in \M) \land \forall i,j(\M_i \oplus \M_j \in \M) \land \forall i \forall X \leq_{\T} \M_i (X \in \M).
  \end{align*}
  Here, $\TJ,\oplus$ and $\leq_{\T}$ denote the Turing jump, the Turing sum and the Turing reduction respectively.
\end{definition}
For $S \subseteq \mathcal{P}(\omega)$, $S$ is a jump ideal if and only if $S$ is an $\omega$-model of $\ACAo$.
This fact is provable in $\ACAo$.
However, a jump ideal may not be a model of $\ACAo$ in $\RCAo$ because 
$\RCAo$ does not ensure the existence of a (partial) truth valuation for a coded $\omega$-model.
We introduce realizers for jump ideals that make a jump ideal an $\omega$-model of $\ACAo$ within $\RCAo$.
\begin{definition}[$\RCAo$]
  Let $\M$ be a jump ideal, $f_{\TJ}$ be a function from $\N$ to $\N$ and $f_{\oplus}$ be a function from $\N^2$ to $\N$.
  We say the pair $(f_{\TJ},f_{\oplus})$ realizes that $\M$ is a jump ideal (written $(f_{\TJ},f_{\oplus}) \forces_{\M}  \ACAo$) if 
  \begin{align*}
    \forall i (\M_{f_{\TJ}(i)} = \TJ(\M_i))  \land 
    \forall i,j(\M_{f_{\oplus}(i,j)} = \M_i \oplus \M_j).
  \end{align*}
  We say a jump ideal $\M$ is realizable if there is a realizer $(f_{\TJ},f_{\oplus})$ for it.
\end{definition}
\begin{remark}
  We note that realizable jump ideals are essentially the same as effective $\omega$-models of $\ACAo$ introduced in \cite{suzuki_yokoyama_fp} in the context of Weihrauch degrees.
\end{remark}

\begin{definition}
  For each $n \in \omega$ such that $n > 0$, 
  let $\pi^0_n$ denote a $\Pi^0_n$-unversal formula.
  That is, for any $\Pi^0_n$ formula $\theta$ whose G\"odel number is  $\GN{\theta}$, the following holds.
  \begin{align*}
    \RCAo \vdash \forall\vec{x},\vec{X} (\theta(\vec{x},\vec{X}) \leftrightarrow \pi^0_n(\GN{\theta},\vec{x},\vec{X}).
  \end{align*}
\end{definition}
  Let $n \in \omega$ such that $n > 0$. Then, $\RCAo$ proves that any realizable jump ideal has the $\Pi^0_n$ truth valuation. More precisely,
  the following holds.
\begin{proposition}
  Let $n \in \omega$ such that $n > 0$. Then, $\RCAo$ proves the following.
  For any  jump ideal $\M$ realized by $f_{\TJ}$ and $f_{\oplus}$,
  there is a $\{\top,\bot\}$-valued function $f\leq_{\T} \M \oplus f_{\TJ} \oplus f_{\oplus }$ such that 
  \begin{align*}
    f(e,x,i,j) = \top \leftrightarrow \pi^0_n(e,x,\M_i,\M_j).
  \end{align*}
\end{proposition}
\begin{proof}
  We first note that $\RCAo$ proves that any $\Pi^0_n(X \oplus Y)$ set is many-one reducible to $\TJ^n(X \oplus Y)$ where $\TJ^n$ denotes the $n$-times iteration of the jump operator, and $\Pi^0_n(X\oplus Y)$ denotes the collection of $\Pi^0_n$-definable sets from an oracle $X\oplus Y$. More precisely, there is a primitive recursive function $p$ such that the following holds in $\RCAo$.
  For any set $X$ and $Y$ such that $\TJ^{n}(X \oplus Y)$ exists, we have
  \begin{align*} 
    \forall e,x (\pi^0_n(e,x,X,Y) \leftrightarrow p(e,x) \in \TJ^n(X\oplus Y)).
  \end{align*}

  We work in $\RCAo$ and show that any realizable jump ideal has the $\Pi^0_n$ truth valuation. Let $\M$ be a jump ideal realized by $(f_{\TJ},f_{\oplus})$.
  Then, for any $e,x,i,j$, $\pi^0_n(e,x,\M_i,\M_j)$ is equivalent to 
  $p(e,x) \in \M_{(f_{\TJ})^n(f_{\oplus}(i,j))}$.
  Therefore, the desired $f$ is $\Delta^0_1$-definable as follows.
  \begin{align*}
    f(e,x) = \top \leftrightarrow p(e,x) \in \M_{(f_{\TJ})^nf_{\oplus}(i,j)}.
  \end{align*}
  This completes the proof.
\end{proof}

For a class of formulas $\Gamma$, we define the class $\exists X \Gamma$ by 
$\{\exists X \varphi : \varphi \in \Gamma\}$. Similarly, we define $\exists f \Gamma$ by $\{\exists f \in \N^{\N} \varphi : \varphi \in \Gamma \}$. Here, 
$f \in \N^{\N}$ is the abbreviation of the $\Pi^0_2$ formula stating in $\RCAo$ that $f$ is a function from $\N$ to $\N$.
It is well-known that there exists a  primitive recursive function $p$ such that 
\begin{align*}
  \RCAo \vdash \exists X \pi^0_2(e,X) \leftrightarrow \exists f \pi^0_1(p(e),f).
\end{align*}
Therefore, we may say that $\exists X \Pi^0_2$ is included in $\exists f \Pi^0_1$. 
Since $\exists f \Pi^0_1$ is a subset of $\exists X \Pi^0_2$, we may think that $\exists X \Pi^0_2$ and $\exists f \Pi^0_1$ are the same. Similarly, we may say that 
$\exists X \Pi^0_2$ and $\exists f \Pi^0_2$ are the same.
Henceforth, we identify these three classes.

\begin{remark}
For $n >0$, $\exists X \Pi^0_n$ is also written as $\Sigma^1_1(\Pi^0_n)$. Similarly, a formula of the form $\lnot \varphi$ for some $\varphi \in \Sigma^1_1(\Pi^0_n)$ is said to be $\Pi^1_1(\Sigma^0_n)$.
  Formally, we inductively define $\Pi^1_{m}(\Gamma)$ and $\Sigma^1_{m}(\Gamma)$ for $\Gamma \in \{\Sigma^0_n,\Pi^0_n\}$ as follows.
  \begin{align*}
    \Sigma^1_{m+1}(\Gamma) = \exists X \Pi^1_n(\Gamma), \\
    \Pi^1_{m+1}(\Gamma) = \forall X \Sigma^1_{m}(\Gamma).
  \end{align*}
\end{remark}

\begin{definition}[$\RCAo$]
  For a set $X$, we define the hyperjump $\HJ(X)$ of $X$ by $\HJ(X) = \{(e,n) : \exists Y \pi^0_2(e,n,X,Y)\}$. 
\end{definition}

\begin{remark}
  In the context of computability theory, $\HJ(X)$ is defined as a certain $\Pi^1_1(X)$-complete set. However, we prefer to define $\HJ(X)$ as a $\Sigma^1_1(\Pi^0_2)$-complete set over $\RCAo$ for technical reasons. We note that $\RCAo$ is not sufficient to prove that $\HJ(X)$ is $\Sigma^1_1$-complete. $\ACAo$ is needed to prove the $\Sigma^1_1$-completeness of hyperjumps.
\end{remark}

\begin{definition}[$\RCAo$]
  Let $\M$ be a coded $\omega$-model. 
  We say $\M$ is a $\beta$-model if 
  \begin{align*}
    \exists X \pi^0_2(e,x,X,\M_i,\M_j) \leftrightarrow \M \models \exists X \pi^0_2(e,x,X,\M_k,\M_i)
  \end{align*}
  holds.
\end{definition}

\begin{remark}
  In the usual context, a $\beta$-model is a $\Sigma^1_1$-elementary submodel of the ground model. Similarly to hyperjumps, $\RCAo$ is not enough to prove the $\Sigma^1_1$-elementarity of $\beta$-models.
\end{remark}

It is well-known that the existence of the hyperjump of $X$ is equivalent to the existence of a coded $\beta$-model containing $X$ over $\ACAo$ \cite{Simpson}.
If one reads this proof carefully, one can find that the construction of a coded $\beta$-model from a hyperjump involves information about $\Sigma^1_1$-truth of the constructed model. To clarify this fact, we introduce a realizer for coded models.
\begin{definition}
  Let $\M$ be a coded $\omega$-model.
  We say a function $f$ is a realizer for $\exists X \Pi^0_2$ sentences true in $\M$ (written $f \forces_{\M} \exists X \Pi^0_2$) 
  if $\exists j \pi^0_2(e,x,\M_i,\M_j) \to \pi^0_2(e,x,\M_i,\M_{f(e,x,i)})$.
\end{definition}
As we will see in the proof of Lemma \ref{Lem LPP implies betaRef}, $\HJ(X)$ computes a coded $\beta$-model $\M$ which contains $X$ and has a computable realizer for $\exists X \Pi^0_2$ sentences.

In \cite{suzuki_yokoyama_fp}, the author and Yokoyama introduced a weaker variant of $\beta$-models called $\Delta^0_k\beta$-model to study pseudo leftmost path principles from the point of view of Weihrauch degrees.
\begin{definition}
  Let $n \in \omega$ such that $n > 0$.
  Let $\M$ be a jump ideal. We say $\M$ is a $\Delta^0_k\beta$-model if for any $\Pi^0_2$ formula $\varphi$ with parameters from $\M$, the following holds.
  \begin{align*}
    \exists X \leq_{\T} \TJ^{k-1}(\M) \varphi(X) \to \exists i \varphi(\M_i)
  \end{align*}
\end{definition}
In the next section, 
we show the equivalence of $\Sigma^0_k\LPP$ and the existence of $\Delta^0_k\beta$-models with computable realizers.

\section{The equivalences}
In this section, we show the equivalence of pseudo leftmost path principles and the $\omega$-model reflection of transfinite inductions.
We first introduce pseudo leftmost path principles.

\begin{definition}
  Let $T \subseteq \N^{<\N}$ be an ill-founded tree.
  \begin{itemize}
    \item   We say a path $f$ of $T$ is the leftmost path (written $\LP(T,f)$) if $f$ is the smallest in $[T]$ with respect to the lexicographic order. Here, $[T]$ denotes the set of paths of $T$.
    \item Let $\alpha$ be a well-order. 
    We say a path $f$ of $T$ is a $\Sigma^0_{\alpha}$-leftmost path if $f$ is the smallest in $[T] \cap \{g : g \in \Sigma^{0}_{\alpha}(T \oplus f)\}$ with respect to the lexicographic order.
  \end{itemize}
\end{definition}

\begin{definition} \cite{Towsner_TLPP,suzuki_yokoyama_pi12}
  We define $\Sigma^0_{n}\LPP$, $\ALPP$ and $\TLPP$ as follows.
  \begin{itemize}
    \item $\Sigma^0_{n}\LPP$ is the assertion that any ill-founded tree has a $\Sigma^0_n$-leftmost path,
    \item $\ALPP$ is the assertion that any ill-founded tree has a $\Sigma^0_{\omega}$-leftmost path,
    \item $\TLPP$ is the assertion that for any well-order $\alpha$ and ill-founded tree $T$, $T$ has a $\Sigma^0_{\alpha}$ leftmost path.
  \end{itemize}
\end{definition}

We note that a path of a tree is a total function. Therefore, a path $g \in [T]$ is $\Sigma^0_n$-definable from $T \oplus f$ if and only if 
it is $\Delta^0_n$-definable from $T \oplus f$. 
\begin{definition}
  We say a path $f \in [T]$ is a $\Delta^0_k$ leftmost path  if 
  $f$ is the smallest in $\{g \in [T] : g \leq_{\T} \TJ^{k-1}(T \oplus f)\}$ with respect to the lexicographic order.
  We define $\Delta^0_k\LPP$ as the assertion that any ill-founded tree has a $\Delta^0_k$-leftmost path. 
\end{definition}
\begin{lemma}
  Over $\RCAo$, $\Sigma^0_k\LPP$ is equivalent to $\Delta^0_k\LPP$ for $k \in \omega$ such that $k > 0$.
\end{lemma}
\begin{proof}
  As we have mentioned, a path is $\Sigma^0_k$-definable if and only if it is $\Delta^0_k$-definable. Therefore, a path is $\Sigma^0_k$-leftmost if and only if it is $\Delta^0_k$-leftmost.
\end{proof}

As it is proved in \cite{suzuki_yokoyama_fp}, $\Delta^0_k\LPP$ is equivalent to the assertion that 
any ill-founded tree has a path $f$ that is the smallest in $\{g \in [T] : g \leq_{\T} \TJ^{k-1}(f)\}$. Thus, we do not distinguish between $\Sigma^0_k\LPP,\Delta^0_k\LPP$ and this assertion.

For studying $\Pi^1_2$ theorems above $\TLPP$, the author and Yokoyama introduced the notion of pseudo hyperjumps in \cite{suzuki_yokoyama_pi12}.
\begin{definition}
  Let $\sigma$ be a sentence and $n \in \omega$. We define the $n$-th pseudo hyperjump with the base $\sigma$ as 
  \begin{align*}
    \beta^1_0\RFN(n;\sigma) \equiv \forall X \exists \M (X \in \M \land \M \models \ACAo + \sigma + \exists \HJ^n(X)).
  \end{align*}
\end{definition}
\begin{theorem}\label{Thm summarly} \cite{suzuki_yokoyama_pi12}
  Over $\ACAo$,
  \begin{align*}
    \beta^1_0\RFN(1;\top) < \TLPP < \beta^1_0\RFN(1;\ATRo) < \beta^1_0\RFN(2;\top).
  \end{align*}
  Here, $T < T'$ means that $T'$ proves $T$ and its consistency.
\end{theorem}

We now provide a new characterization of $\TLPP$ via a pseudo hyperjump.
First, we introduce some properties of $\TLPP$.

\begin{lemma}[$\TLPP$] \cite[Theorem 4.6.]{Towsner_TLPP}
For any ill-founded tree $T$, there exists an $\omega$-model of $\Sigma^1_1\hyp\DC_0$ which believes that 
$T$ has a leftmost path.
\end{lemma}

\begin{lemma}[$\RCAo$]\label{Lem Jump closed} \cite[Lemma 4.23.]{suzuki_yokoyama_pi12}
$\TLPP$ is equivalent to the assertion that for any well-order $\alpha$ and set $X$, there exists an $\omega$-model $\M$ such that 
$\alpha,X \in \M$ and $\M$ believes $\HJ(X)$ exists and $\M$ is closed under $\alpha$-times Turing jump.
\end{lemma}

To connect leftmost paths and hyperjumps, we use the following lemma.
\begin{lemma}
  Over $\RCAo$, we may assume that there exists a total Turing functional $\Phi_T(X)$ such that
  \begin{itemize}
    \item $\Phi_T(X)$ is a tree,
    \item $\Phi_T(X)$ have a infinite path which is uniformly computable from $X$,
    \item each path of $\Phi_T(X)$ computes $X$,
    \item if $\LP(\Phi_T(X),f)$ holds, then $\HJ(X) \leq_{\T} f$.
  \end{itemize}
\end{lemma}
  For details, see \cite[Section 4]{suzuki_yokoyama_pi12}.
  We note that in \cite{suzuki_yokoyama_pi12}, the above properties are proved in $\ACAo$ rather than $\RCAo$. However, if one looks at the proofs carefully, one can find that these properties hold in $\RCAo$.

\begin{theorem}
  Over $\RCAo$, $\TLPP$ is equivalent to $\beta^1_0\RFN(1;\Pi^1_1\hyp\TI)$.
\end{theorem}
\begin{proof}
Since $\Pi^1_1\hyp\TI_0$ is equivalent to $\Sigma^1_1\hyp\DC_0$ over $\ACAo$, 
it is enough to show the equivalence of $\TLPP$ and  $\beta^1_0\RFN(1;\Sigma^1_1\hyp\DC)$.

We first show that $\TLPP$ implies $\beta^1_0\RFN(1;\Sigma^1_1\hyp\DC)$.
Suppose $\TLPP$.
Let $X$ be a set.
Then, there exists an $\omega$-model $\M$ of $\Sigma^1_1\hyp\DC_0$ such that 
$\M$ believes $\Phi_T(X)$ has a leftmost path.
Therefore, $\M$ satisfies $\Sigma^1_1\hyp\DC_0$ and there exists the hyperjump of $X$.

We then show the converse direction by Lemma \ref{Lem Jump closed}.
Let $\alpha$ be a well-order and $X$ be a set.
Take an $\omega$-model $\M$ such that $\alpha,X \in \M$ and $\M \models \Sigma^1_1\hyp\DC_0$.
Then, $\M$ is closed under $\alpha$-times Turing jump.
\end{proof}

\begin{remark}
  We note that $\Pi^1_1\hyp\TI$ can be weaken to Montalb\'{a}n's Jump Iteration $\mathsf{JI}$ \cite{montalban2006indecomposable} in the above proof.
\end{remark}

\begin{corollary}\label{Cor: pseudo Pi^1_1-TI < pseudo ATR}
  Over $\RCAo$, $\beta^1_0\RFN(1;\Pi^1_1\hyp\TI) < \beta^1_0\RFN(1;\ATRo)$.
\end{corollary}
\begin{proof}
  It is immediate from the equivalence of $\beta^1_0\RFN(1;\Pi^1_1\hyp\TI_0)$ and $\TLPP$.
\end{proof}

We next show that both $\ALPP$ and the single pseudo hyperjump $\beta^1_0\RFN(1;\top)$ are characterized by the $\omega$-model reflection of full transfinite induction.
We note that this characterization is independently proved by Freund \cite{Freund}.

\begin{theorem}\label{thm InftyTI and RFN}
  Over $\RCAo$, the following assertions are equivalent.
  \begin{enumerate}
    \item $\ALPP$,
    \item $\beta^1_0\RFN(1)$,
    \item $\RFN(\Pi^1_{\infty}\hyp\TI)$,
    \item $\RFN(\mathrm{A}\Pi^1_1\hyp\TI)$ where $\mathrm{A}\Pi^1_1\hyp\TI_0$ is the class built from $\Pi^1_1$ with logical connectives and number quantifiers.
  \end{enumerate}
\end{theorem}
\begin{proof}
  The equivalence of $(1)$ and $(2)$ is proved in \cite{suzuki_yokoyama_pi12}, and the equivalence of $(2)$ and $(3)$ is proved in \cite{Freund}.
  The implication of $(3)$ to $(4)$ is trivial.
  For the implication of $(4)$ to $(2)$, the same proof of the implication of $(3)$ to $(2)$ works. For this implication, see also \cite[Theorem VII.2.18 and Corollary VII.2.21]{Simpson}.
\end{proof}

The equivalence of $(2)$ and $(3)$ in the previous theorem can be extended to the $n$-th pseudo hyperjump level as follows.
\begin{theorem}\label{thm: n+1 equiv n + TI}
  Over $\RCAo$, $\beta^1_0\RFN(n+1)$ is equivalent to $\RFN(n;\Pi^1_{\infty}\hyp\TI)$ for any $n \in \omega$ such that $n > 0$.
\end{theorem}
\begin{proof}
  Since the same proof for the previous theorem works, we give a brief proof. The implication from $\beta^1_0\RFN(n+1)$ to $\RFN(n;\Pi^1_{\infty}\hyp\TI)$ is trivial.
  For proving $\beta^1_0\RFN(n+1)$ from $\RFN(n;\Pi^1_{\infty}\hyp\TI)$, take an $X$ and 
  an $\omega$-model $\M$ such that $X \in \M \land \M \models (\exists \HJ^n(X) \land \Pi^1_{\infty}\hyp\TI)$. Let $Y \in \M$ such that $\M \models Y = \HJ^n(X)$. Let $\M'$ be the collection of the definable-in-$\M$ subclasses of $\N$. Then, $\M \subseteq_{\beta} \M'$. Moreover, 
there is a set $Z \in \M'$ such that $Z = \{n : \M \models \theta_{\HJ}(n,Y)\}$ where $\theta_{\HJ}$ is the $\Sigma^1_1$ formula defining hyperjumps. Then, $\M' \models Z = \HJ(Y) \land Y = \HJ^n(X)$, hence $\M' \models Z = \HJ^{n+1}(X)$. This completes the proof.
\end{proof}

Combining Corollary \ref{Cor: pseudo Pi^1_1-TI < pseudo ATR} and the previous theorem, we have 
\begin{corollary}
  Over $\RCAo$,
  \begin{align*}
    \beta^1_0\RFN(1;\top) < \beta^1_0\RFN(1;\Pi^1_1\hyp\TI) < \beta^1_0(1;\Sigma^1_1\hyp\TI) < \beta^1_0\RFN(1;\Pi^1_{\infty}\hyp\TI).
  \end{align*}
\end{corollary}
\begin{proof}
  We first remark that any of the above statements includes $\ACAo$. 
  We also note that $\beta^1_0(1;\Sigma^1_1\hyp\TI)$ is equivalent to $\beta^1_0(1;\ATRo)$ because $\Sigma^1_1\hyp\TI$ and $\ATRo$ are equivalent over $\Sigma^1_1$-induction that is true in any coded $\omega$-model.
  Thus, the inequality $\beta^1_0\RFN(1;\top) < \beta^1_0\RFN(1;\Pi^1_1\hyp\TI) < \beta^1_0(1;\Sigma^1_1\hyp\TI)$ follows from Theorem \ref{Thm summarly} and Corollary \ref{Cor: pseudo Pi^1_1-TI < pseudo ATR}. The inequality $\beta^1_0(1;\Sigma^1_1\hyp\TI) < \beta^1_0\RFN(1;\Pi^1_{\infty}\hyp\TI)$ follows from Theorem \ref{Thm summarly} and Theorem \ref{thm: n+1 equiv n + TI}.
\end{proof}

We then show that the equivalence in Theorem \ref{thm InftyTI and RFN} also holds at the level of $\Sigma^0_k\LPP$. 
For proving the equivalence of $\Sigma^0_k\LPP$ and $\RFN(\Pi^1_{k+1}\hyp\TI)$, we use $\Delta^0_k\beta$-models with computable realizers.

\begin{definition}
  Let $n \in \omega$ such that $n > 0$. We define $\Ef\Delta^0_n\beta\hyp\RFN$ as the assertion that 
  any set $X$ is contained in a $\Delta^0_n\beta$-model $\M$ with total computable functions $f_{\TJ},f_{\oplus},f_{\Sigma^1_1}$ such that 
  $(f_{\TJ},f_{\oplus}) \forces_{\M} \ACAo$ and 
  $f_{\Sigma^1_1} \forces_{\M} \exists X \Pi^0_2$.
\end{definition}

\begin{lemma}\label{Lem LPP implies betaRef}
 For $k \in \omega$ such that $k > 0$, $\Sigma^0_{k+1}\LPP$ implies $\Ef\Delta^0_{k+1}\beta\hyp\RFN$ over $\RCAo$.
\end{lemma}
\begin{proof}
  We first note that $\Sigma^0_{k+1}\LPP$ includes $\ACA_0^+$. Thus, we may assume that any coded $\omega$-model has a total valuation.

  Let $X$ be a set. Take a $\Delta^0_{k+1}$-leftmost path $f$ of $\Phi_T(X)$.
  Define an $\omega$-model $\M = \Delta^0_{k+1}(\Phi_T(X) \oplus f) = \{A : A \leq_{\T} \TJ^{k}(\Phi_T(X) \oplus f)\}$.
  Then, $\M \models \RCAo +  \LP(\Phi_T(X),f)$.
  Therefore, $\HJ^{\M}(X) \leq_{\T} f$ exists.

  We perform the construction of \cite[VII.2.9 (19)]{Simpson} within $\M$.
  Then, we have a function $g \leq_{\T} \HJ^{\M}(X) \leq_{\T} f$ having the following condition.
  \begin{align*}
    \M \models \forall e,m (\exists h \in \N^{\N} \pi^0_1(e,m,(g)_{<\langle e,m \rangle},h) \to \pi^0_1(e,m,(g)_{<\langle e,m \rangle},g_{\langle e,m \rangle})).
  \end{align*}
  Here, $g_n$ denotes the $n$-th segment of $g$, that is, $g_{n}(x) = g(\langle n,x\rangle)$.
  Thus, $g$ is identified the sequence $\langle g_n \rangle_n$.
  In addition, $g_{<n}$ denotes the sequence $\langle g_0 \ldots, g_{n-1} \rangle$.
   
  We show that the sequence $\langle g_n \rangle_n$ is closed under Turing jumps. That is, for any $n$, there exists $m$ such that $g_m = \TJ(g_n)$.
  To see this, take an $n$. Then, $g_n \leq_{\T} g \leq_{\T} f$ by definition.
  Thus, $\TJ(g_n) \leq_{\T} \TJ(f)$ and hence $\TJ(g_n) \in \M$.
  Since $\exists h(h = \TJ(g_n))$ is of the form $\exists h' \in \N^{\N} \Pi^0_1$, we have $g_{m_n} = \TJ(g_n)$ for some $m_n$.
  Moreover, the correspondence $f_{\TJ} : n \mapsto m_n$ is clearly computable.
  By a similar way, we have that there exists a computable correspondence $f_{\oplus} : (n,n') \mapsto m'(n,n')$ such that
  $g_{m'(n,n')} = g_n \oplus g_n'$. It is also proved by a similar way that $\langle g_n \rangle_n$ is closed under the Turing reduction.
  Therefore, $\langle g_n \rangle_n$ is a jump ideal realized by computable functions $f_{\TJ}$ and $f_{\oplus}$.
  In addition, $\langle g_n \rangle_n$ and $\M$ are clearly absolute for $\exists h \in \N^{\N}\Pi^0_1$ formulas.

  We show that $\langle g_n \rangle_n$ is a $\Delta^0_{k+1}\beta$-model.
  Let $e,x,m \in \N$ and $h \leq_{\T} \TJ^{k}(\langle g_n \rangle_n)$ be such that  $\pi^0_1(e,x,g_m,h)$.
  Then, $h \leq_{\T} \TJ^{k}(\Phi_T(X) \oplus f)$ because $ \langle g_n\rangle_n \leq_{\T}  f$.
  Hence, $h \in \M$. Thus we have
  $\M \models \exists h \pi^0_1(e,x,g_n,h)$.
  By the construction of $\langle g_n \rangle_n$, $\langle g_n \rangle_N \models \pi^0_1(e,x,g_m,g_{m'})$ for some $m'$.
  Consequently, we have that $\langle g_n \rangle_n$ is a $\Delta^0_{k+1}\beta$-model.

  We then show that $\langle g_n \rangle_n$ has a computable realizer for $\exists h \in \N^{\N} \Pi^0_1$ formulas.
  We note that 
  \begin{align*}
   \langle g_n \rangle_n \models \exists h \pi^0_1(e,m,g_{<\langle e,m \rangle},h) \text{ if and only if }
  \langle g_n \rangle_n \models \pi^0_1(e,m,g_{<\langle e,m \rangle},g_{\langle e,m \rangle})
  \end{align*}
  holds.
  Hence, to decide either $\langle g_n \rangle_n \models \exists h \pi^0_1(e,m,g_{<\langle e,m \rangle},h)$ or not, it is enough
  to check either $\langle g_n \rangle_n \models \pi^0_1(e,m,g_{<\langle e,m \rangle},g_{\langle e,m \rangle})$ or not.
  Since $\langle g_n \rangle_n$ is a realizable jump ideal, it has a computable truth valuation for $\Pi^0_1$ formulas. Thus, we have a computable function which decides either $\langle g_n \rangle_n \models \pi^0_1(e,m,g_{<\langle e,m \rangle},g_{\langle e,m \rangle})$ or not. This completes the proof.
\end{proof}

\begin{lemma}\label{Lem definability}
  Let $\M$ be a $\Delta^0_k\beta$-model and $f_{\TJ},f_{\oplus},f_{\Sigma^1_1}$ be computable functions such that 
  $(f_{\TJ},f_{\oplus}) \forces_{\M}  \ACAo$ and 
  $f_{\Sigma^1_1} \forces_{\M} \exists X \Pi^0_2$.
  Then, for any arithmetical formula $\varphi(n,X,Y)$ and $n,i \in \N$,
  the set $\{n : \M \models \exists Y\varphi(n,\M_i,Y)\}$ is $\M$-computable.
\end{lemma}
\begin{proof}
  We may assume that $\varphi$ is $\Pi^0_2$ because $\M$ is a model of $\ACAo$.

  Let $e$ be a code of $\varphi$.
  Since $f_{\Sigma^1_1} \forces_{\M} \exists X \Pi^0_2$, 
  $\{n : \M \models \exists Y \varphi(n,\M_i,Y)\} = \{n : \varphi(n,\M_i,\M_{f_{\Sigma^1_1}(e,n,i)))}\}$.
  Since $\M$ is a realizable jump ideal, there is a $(\M \oplus f_{\TJ} \oplus f_{\oplus})$-computable truth valuation for $\Pi^0_2$ formulas. 
  Thus, $\{n: \varphi(n,\M_i,\M_{f_{\Sigma^1_1}(e,n,i)})\}$ is computable from $\M  \oplus f_{\TJ} \oplus f_{\oplus} \oplus f_{\Sigma^1_1}$. Since the realizers are computable, $\{n : \varphi(n,\M_i,\M_{f_{\Sigma^1_1}(e,n,i)})\}$ is $\M$-computable.
\end{proof}

\begin{lemma}
  Let $k \in \omega$ such that $k > 0$.
  Let $\M$ be a $\Delta^0_k\beta$-model having computable realizers $f_{\TJ},f_{\oplus},f_{\Sigma^1_1}$.
  Then, for any $\Sigma^1_{k+1}$ formula $\varphi(n,X)$, the set $\{n: \M \models \varphi(n,\M_i)\}$ is $\Sigma^0_{k}$-definable from $\M$ and $\{n : \M \models \lnot \varphi(n,\M_i)\}$ is $\Pi^0_k$-definable from $\M$.
\end{lemma}
\begin{proof}
  It is an easy induction on $k$.
\end{proof}

\begin{lemma}[$\ACAo$]\label{Lem betamodel is a model of TI}
  Let $k \in \omega$ such that $k > 0$.
  Let $\M$ be a $\Delta^0_k\beta$-model, $f_{\TJ},f_{\oplus}$ $f_{\Sigma^1_1}$ be computable realizers for $\M$.
  Then, $\M \models \Pi^1_{k+1}\hyp\TI$.
\end{lemma}
\begin{proof}
  Let $W \in \M$ be such that $\M \models \WO(W)$.
  Let $\varphi(i)$ be a $\Pi^1_{n+1}$ formula such that $\M$ satisfies the antecedent
of transfinite induction along $W$. That is,
  \begin{align*}
    \M \models \forall i \in |W| ( \forall j <_W i \varphi(j) \to \varphi(i))
  \end{align*}
  where $|W|$ denotes the field of $W$.
  
  Assume that $\M \not \models \forall i \in |W| \varphi(i)$.
  Then, $\{i \in |W| : \lnot \varphi(i)\}$ is nonempty and has no $W$-minimal element. 
  In addition, $\{i \in |W| : \lnot \varphi(i)\}$ is $\Sigma^0_k$-definable from $\M$. 
  Thus, there exists a $W$-descending sequence from 
  $\{i \in |W| : \lnot \varphi(i)\}$ which is $\Delta^0_k$ definable from $\M$. That is, 
  we have 
  \begin{align*}
  \exists A \leq_{\T} \TJ^{k-1}(\M) (A \text{ is a descending sequence of }  W).
  \end{align*}
  Since $\M$ is a $\Delta^0_k\beta$-model, $\M \models \lnot \WO(W)$. However, this is a contradiction.
\end{proof}
Now, we have  the implication  $\Sigma^0_{k+1}\LPP \to \RFN(\Pi^1_{k+2}\hyp\TI)$.
For the converse, we consider a $\Delta^0_k$ analogue of the implication $\RFN(\mathrm{A}\Pi^1_1\hyp\TI) \to \ALPP$.
The point of this implication is that for a given coded $\omega$-model $\M$ of $\mathrm{A}\Pi^1_1\hyp\TI$, the set of 
$\mathrm{A}\Pi^1_1\hyp\TI$-definable subclasses $\{ \{n  : \M \models \theta(n)\} : \theta \in \mathrm{A}\Pi^1_1\hyp\TI \}$ of $\M$ forms a $\beta$-extension of $\M$.
For considering a $\Delta^0_k$ analogue, we begin with introducing the classes $\Sigma^0_k(\Pi^1_1),\Pi^0_k(\Pi^1_1)$ and $\Delta^0_k(\Pi^1_1)$.
Intuitively, $\Sigma^0_k(\Pi^1_1)$ is the class of formulas of the form $\varphi[t \in X/\theta(t)]$ for some $\Sigma^0_k$ formula $\varphi$ and $\Pi^1_1$ formula $\theta(x)$, $\Pi^0_k(\Pi^1_1)$ is its dual, and $\Delta^0_k(\Pi^1_1)$ is the intersection of $\Sigma^0_k(\Pi^1_1)$ and $\Pi^0_k(\Pi^1_1)$. Here, $\varphi[t \in X/\theta(t)]$ is the formula obtained from $\varphi$ by replacing each occurrence of $t \in X$ with $\theta(t)$. Formally, we define these classes as follows.
\begin{definition}
  We  say a set $\varphi$ is $\Pi^0_1(\Pi^1_1)$ if it is of the form 
  $\forall x \theta(x)$ for some formula $\theta$ which is built from $\Pi^1_1$ using logical connectives $\lor,\land,\lnot$ and bounded quantifiers $\forall x < y,\exists x < y$.
  Similarly, a formula is $\Sigma^0_1(\Pi^1_1)$  if it is of the form 
  $\exists x \theta(x)$ for some formula $\theta$ which is built from $\Pi^1_1$ using logical connectives $\lor,\land,\lnot$ and bounded quantifiers $\forall x < y,\exists x < y$.
  For the general case, we say a formula is $\Pi^0_{k+1}(\Pi^1_1)$ if it is of the form $\forall x \theta(x)$ for some formula $\theta$ which is built from $\Pi^0_k(\Pi^1_1)$ using logical connectives $\lor,\land,\lnot$ and bounded quantifiers $\forall x < y,\exists x < y$.
  We also say a formula is $\Sigma^0_{k+1}(\Pi^1_1)$  if it is of the form 
  $\exists x \theta(x)$ for some formula $\theta$ which is built from $\Pi^0_k(\Pi^1_1)$ using logical connectives $\lor,\land,\lnot$ and bounded quantifiers $\forall x < y,\exists x < y$.
 
  We say a formula is $\Delta^0_n(\Pi^1_1)$ (in a fixed base theory $T$) if it belongs to both $\Pi^0_n(\Pi^1_1)$ and $\Sigma^0_n(\Pi^1_1)$.
\end{definition}

In what follows, we show that the class of $\Delta^0_n(\Pi^1_1)$-definable sets forms a model of $\RCAo$ in a sense.
\begin{lemma}[$\RCAo + \Sigma^1_{\infty}\hyp\IND$]
  The class $\Sigma^0_n(\Pi^1_1)$ is closed under $\lor,\land$ and bounded quantifications and the existential number quantification. Similarly, $\Pi^0_n(\Pi^1_1)$ is closed under $\lor,\land$ and bounded quantifications and the universal number quantification.
\end{lemma}
\begin{proof}
Working in $\RCAo + \Sigma^1_{\infty}\hyp\IND$, 
the usual proof for the closure properties of $\Sigma^0_n$ and $\Pi^0_n$ works.
\end{proof}

\begin{lemma}[$\RCAo + \Sigma^1_{\infty}\hyp\IND$] \label{Lem Pi0n(Pi11) and Pi0n+1}
  The class $\Pi^0_n(\Pi^1_1)$ is a subclass of $\Pi^1_{n+1}$.
\end{lemma}
\begin{proof}
  We prove by induction on $n$.
  We note that if a formula $\varphi$ is built from $\Pi^1_1$ with $\land,\lor,\lnot,\forall x < y,\exists x < y$, then $\varphi$ is $\Delta^1_2$ in $\RCAo + \Sigma^1_{\infty}\hyp\IND$. 
  Thus, a $\Pi^0_1(\Pi^1_1)$ formula is of the form $\forall x \Delta^1_2$, and hence $\Pi^1_2$.

  For the induction step, we can show that a formula built from $\Pi^0_n(\Pi^1_1)$ with $\land, \lor, \lnot,\forall x < y,\exists x < y$ is $\Delta^1_{n+1}$. Thus, we have that a $\Pi^0_{n+1}(\Pi^1_1)$ formula is $\Pi^1_{n+1}$ formula.
\end{proof}

\begin{lemma}[$\RCAo + \Sigma^1_{\infty}\hyp\IND$]
  Let $\varphi(n)$ be a $\Sigma^0_k(\Pi^1_1)$ formula and $\psi(n)$ be a $\Pi^0_k(\Pi^1_1)$ formula. If $X = \{n : \varphi(n)\} = \{n : \psi(n)\}$ and $Y \leq_{\T} X$, then $Y$ is $\Delta^0_n(\Pi^1_1)$-definable.
\end{lemma}
\begin{proof}
  It is enough to show that if $Y$ is $\Sigma^0_1$-definable from $X$ then $Y$ is $\Sigma^0_k(\Pi^1_1)$-definable, and if $Y$ is $\Pi^0_1$-deinable from $X$ then $Y$ is $\Pi^0_k(\Pi^1_1)$-definable. This is immediate from induction on the construction of $\Sigma^0_1/\Pi^0_1$ formulas.
\end{proof}

\begin{lemma}[$\RCAo+ \Sigma^1_{\infty}\hyp\IND$]
  Let $\varphi_X(n),\varphi_Y(n)$ be $\Sigma^0_k(\Pi^1_1)$ formulas and $\psi_X(n)$, $\psi_Y(n)$ be $\Pi^0_k(\Pi^1_1)$ formulas.
  If $X = \{n: \varphi_X(n)\} = \{n: \varphi_Y(n)\}$ and $Y = \{n : \varphi_Y(n)\} = \{n : \varphi_Y(n)\}$,
  then $X \oplus Y$ is $\Delta^0_k(\Pi^1_1)$-definable.
\end{lemma}
\begin{proof}
  Recall that $X \oplus Y = \{2n : n \in X\} \cup \{2n+1 : n \in Y\} = \{2n : \varphi_X(n)\} \cup \{2n+1 : \varphi_Y(n)\}$.
  Define $\varphi(n) \equiv \exists m ((m = 2n \land \varphi_X(m)) \lor (m = 2n+1 \land \varphi_Y(m))$. Since $\Sigma^0_k(\Pi^1_1)$ is closed under logical connectives and existential number quantification, $\varphi$ is also $\Sigma^0_k(\Pi^1_1)$.
  Similarly, $\psi(n) \equiv \forall m ((m = 2n \to \varphi_X(m)) \lor (m = 2n+1 \to \varphi_Y(m))$ is $\Pi^0_k(\Pi^1_1)$.
Then, $X \oplus Y$ is defined by both $\varphi$ and $\psi$. This completes the proof.
\end{proof}

\begin{lemma}[$\ACA_0$]\label{Lem conservertivity}
  Let $\M$ and $\NN$ be coded $\omega$-models such that 
  $\M \models \ACAo,\NN \models \RCAo$ and 
  $\forall X(X \in \M \to X \in \NN)$.
  Assume that for any linear order $L \in \M$, $\M \models \WO(L)$ if and only if $\NN \models \WO(L)$.
  Then for any $\Sigma^1_1(\Pi^0_2)$ sentence $\sigma$ with parameters from $\M$, 
  $\M \models \sigma $ if and only if $\NN \models \sigma$.
\end{lemma}
\begin{proof}
  Write $\sigma \equiv \exists X \forall x \exists y \theta$ by a $\Sigma^0_0$ formula $\theta$. It is enough to show that if $\NN \models \sigma$ then $\M \models \sigma$.

  By Skolemization and Kleene's normal form, there is a $\Sigma^0_0$ formula $\eta$ such that $\sigma \leftrightarrow \exists X \exists f \in \N^{\N} \forall x \eta(X[x],x,f[x])$ is provable in $\RCAo$.
  Assume $\NN \models \exists X \forall x \exists y \theta(X,x,y)$. Pick an $X_0$ and $f_0$ such that $\NN \models \forall x  \eta(X_0[x],x,f_0[x])$.
  We define a tree $T \subseteq (2 \times \N)^{<\N}$ by 
  $T = \{(\sigma, \tau) : \eta(\sigma,|\sigma|,\tau) \}$.
  Then, $T$ is computable from parameters in $\sigma$. Thus, $T$ is in $\M$ and $\N$.
  
  We have $\NN \models (X_0,f_0) \in [T]$. Thus, $\NN \models \lnot \WO(\KB(T))$ where $\KB(T)$ denotes the Kleene-Brouwer ordering of $T$.
By assumption, $\M \models \lnot \WO(\KB(T))$. Since $\M$ is a model of $\ACAo$, $\M \models [T] \neq \varnothing$. 
  Thus, $\M \models \sigma$.
\end{proof}

\begin{lemma}[$\RCAo$]
\label{Lem ref of TI implies LPP}
  $\RFN(\Delta^0_k(\Pi^1_1)\hyp\TI_0)$ implies $\Delta^0_k\LPP$.
\end{lemma}
\begin{proof}
  Since $\RFN(\Delta^0_k(\Pi^1_1)\hyp\TI_0)$ implies $\ACA_0^
  +$, we may assume $\ACA_0^+$.

  Let $T$ be an ill-founded tree with a path $f$.
  Take an $\omega$-model $\M$ such that $\M \models \Delta^0_k(\Pi^1_1)\hyp\TI_0$ and $T,f \in \M$.

  Let $\langle (\varphi_e(x),\psi_e(x)) \rangle_e$ be an enumeration of $\Delta^0_k(\Pi^1_1)\hyp\text{in}\hyp\M$ formulas. That is, each $\varphi_e$ is $\Sigma^0_k(\Pi^1_1)$ and $\psi_e$ is $\Pi^0_k(\Pi^1_1)$ having the same parameters from $\M$, and
  $$ \M \models \forall x (\varphi_e(x) \leftrightarrow \psi_e(x)).$$

  Define $\NN_{e} = \{x: \M \models \varphi_e(x)\}$ and $\NN = \langle \NN_e \rangle_e$.
  Since $\NN$ is closed under Turing reduction and Turing sum, $\NN$ is an $\omega$-model of $\RCAo$.
  
  We show that $\M \models \WO(L)$ if and only if $\NN \models \WO(L)$ for any linear order $L \in \M$.
  Let $L \in \M$ be a linear order. Assume $\M \models \WO(L)$ but $\NN \not \models \WO(L)$.
  Then, there exists $Y \in \NN$ such that $\varnothing \subsetneq Y \subseteq |L|$ and $Y$ has no minimal element. Here, $|L|$ denotes the field of $L$.
  Let $\varphi_e,\psi_e$ be the defining formulas of $Y$. That is,
  \begin{align*}
     &\M \models \varphi_e \leftrightarrow \psi_e, \\
     &\forall x (x \in Y \leftrightarrow \M \models \varphi_e(x)).
  \end{align*}
  We show that $\M \models \forall i \in |L| (\lnot \varphi_e(i))$ by transfinite induction along $L$.
  Let $i \in |L|$ such that $\M$ satisfies $\forall j <_L i\,\lnot\varphi_e(j)$.
  Then, $\forall j <_L i (j \not \in Y)$. Since $Y$ has no minimal element, $i \not \in Y$ and hence $\M$ satisfies
  $\lnot \varphi_e(i)$.
  We have $Y = \varnothing$, contradiction.

  We note that $\M$ and $\NN$ are $\exists f \Pi^0_2$ conservative by Lemma \ref{Lem conservertivity}.
  Define a tree $S = \{\sigma : \M \models \exists f \in [T] (\sigma \prec f)\}$.
  Then, $S = \{\sigma : \NN \models \exists f \in [T] (\sigma \prec f)\}$ by conservativity.
  In addition, $S \in \NN$ because $S$ is $\Sigma^1_1$ definable over $\M$.
  Let $\varphi(x)$ be a $\Sigma^1_1$ formula defining $S$.
  Then, there exists a $\Delta^0_{k}(\Pi^1_1)$ formula $\psi(x)$ such that
  $\TJ^{k-1}(S) = \{x : \M \models \psi(x)\}$.
  Thus, $\TJ^{k-1}(S)$ is in $\NN$.

  We note that $S$ is an ill-founded pruned tree by definition. Therefore, $S$ computes its leftmost path $g$.
  Then, $\TJ^{k-1}(g) \leq_{\T} \TJ^{k-1}(S)$ and hence $g$ and $\TJ^{k-1}(g) \in \NN$.
  We claim that $g$ is a $\Delta^0_k$-leftmost path of $T$.
  Indeed, if a path $h \in [T]$ satisfies $h \leq_{\T} \TJ^{k-1}(g)$, then $h \in \NN$ and hence $h \in [S]$.
  Then, $g$ is lexicographically smaller than $h$ because $g$ is the leftmost path of $S$.
\end{proof}

\begin{theorem}
  For $k \in \omega$ such that $k > 1$, the following assertions are equivalent over $\RCAo$
  \begin{enumerate}
    \item $\Sigma^0_k\LPP$,
    \item $\Ef\Delta^0_k\beta\hyp\RFN$,
    \item the $\omega$-model reflection of $\Pi^1_{k+1}\hyp\TI$,
    \item the $\omega$-model reflection of $\Delta^0_k(\Pi^1_1)\hyp\TI$.
  \end{enumerate}
\end{theorem}
\begin{proof}
  We note that each assertion includes $\ACAo$. Therefore, we may assume that we work in $\ACAo$.

  The implication $(1) \to (2)$ is Lemma \ref{Lem LPP implies betaRef}.
  The implication $(2) \to (3)$ is immediate form Lemma \ref{Lem betamodel is a model of TI}.
  The implication  $(3) \to (4)$ follows from Lemma \ref{Lem Pi0n(Pi11) and Pi0n+1}.
  The implication  $(4) \to (1)$ is Lemma \ref{Lem ref of TI implies LPP}.
 \end{proof}

  \begin{remark}
 We note that the implications  $(2) \to (3) \to (4) \to (1)$ still hold for $n =1$.
 We do not know either $(1) \to (2)$ holds for $n=1$.
  \end{remark}

  \begin{remark}
    In \cite{Towsner_TLPP}, Towsner proved that 
    \begin{enumerate}
      \item $\RCAo + \Sigma^0_{k+2}\LPP$ implies the existence of an $\omega$-model of $\Pi^0_{k}(\Pi^1_1)\hyp\TI$ (Theorem 4.3),
      \item $\RCAo + \Pi^0_{k+2}(\Pi^1_1)\hyp\TI$ proves $\Sigma^0_k\LPP$ (Theorem 6.1).
    \end{enumerate}
    The previous theorem can be seen as a refinement of these results.
  \end{remark}

  At the end of this section, we see some applications of the previous theorem.

  \begin{corollary}
    Let $k \in \omega$ such that $k > 1$. 
    Then $\Sigma^0_{k}\LPP$ is equivalent to $\RFN(\Pi^1_{k+2}\hyp\RFN_0)$ over $\RCAo$.
  \end{corollary}
  \begin{proof}
    It is immediate from the fact that $\Pi^1_{k+1}\hyp\TI$ is equivalent to the $\Pi^1_{k+2}$ reflection over $\RCAo$.
  \end{proof}

  \begin{corollary}\label{Cor LPP proves the omega-ref of LPP}
    Let $n \in \omega$ such that $n > 0$.
    Then, $\Sigma^0_{n+1}\LPP$ proves the $\omega$-model reflection of $\Sigma^0_n\LPP$ over $\RCAo$. Hence, $\Sigma^0_{n+1}\LPP$ is strictly stronger than $\Sigma^0_n\LPP$.
  \end{corollary}
  \begin{proof}
    Since $\Pi^1_{n+2}\hyp\RFN_0$ proves a $\Pi^1_{n+2}$ axiomatizable theory $\Pi^1_{n+1}\hyp\RFN_0$, it proves the $\omega$-model reflection of $\Pi^1_{n+1}\hyp\RFN_0$.
    Hence, we have the following implications.
    \begin{align*}
      \Sigma^0_{n+1}\LPP \to \RFN(\Pi^1_{n+2}\hyp\RFN_0)
      \to \RFN(\RFN(\Pi^1_{n+1}\hyp\RFN_0)) \to \RFN(\Sigma^0_{n}\LPP).
    \end{align*}
  \end{proof}

  \begin{remark}
    Our $\RFN(\Pi^1_{n+2}\RFN_0)$ is the same as Probst's $\mathsf{p_1}\mathsf{p_{n+2}}(\ACAo)$.
    He studied the proof theoretic strength of $\mathsf{p_1}\mathsf{p_{n+2}}(\ACAo)$ in his habilitation thesis \cite{Probst}.
  \end{remark}



  \section{Weihrauch degrees}
  In \cite{suzuki_yokoyama_fp}, the author and Yokoyama introduced a Weihrauch problem corresponding to $\Sigma^0_n\LPP$ and studied the computability-theoretic strength of it. In this section, we refine the results in \cite{suzuki_yokoyama_fp}. We prove that $\Sigma^0_{n+1}\LPP$ is strictly stronger than $\Sigma^0_{n}\LPP$ in the sense of Weihrauch reducibility.

  \begin{definition}
    A Weihrauch problem is a partial multivalued function on the Cantor space $2^{\omega}$.
  \end{definition}
  We sometimes identify a Weihrauch problem $\P$ and a partial function $\P'$ from $2^{\omega}$ to $\mathcal{P}(2^{\omega})$ defined by $\P'(X) = \{Y \subseteq \omega: Y \text{ is an output of $\P$ at $X$}\}$. In this sense, we write $Y \in \P(X)$ to mean $Y$ is an output of $\P$ at $X$.

  \begin{definition}
    Let $\P$ and $\Q$ be Weihrauch problems.
    We say $\P$ is Weihrauch reducible to $\Q$ (written $\P \leq_{\W} \Q$) if there are (partial) Turing functionals $\Phi$ and $\Psi$ such that 
    \begin{align*}
      \forall X \in \dom(\P)(\Phi(X)\mathord{\downarrow} \in \dom(\Q) \land 
      \forall Y \in \Q(\Phi(X)) (\Psi(X,Y)\mathord{\downarrow} \in \P(X))
      ).
    \end{align*}
    We write $\P <_{\W} \Q$ to mean 
    $\P \leq_{\W} \Q$ but $\Q \not \leq_{\W} \P$.

    We say $\P$ and $\Q$ are Weihrauch equivalent (written $\P \equiv_{\W} \Q$) if both of $\P \leq_{\W} \Q$ and $\Q \leq_{\W} \P$ hold.
  \end{definition}

  \begin{definition}
      We define Weihrauch problems $\Sigma^0_n\LPP,\Delta^0_n\LPP$ and $\Delta^0_n\overline{\LPP}$ as follows.
    \begin{itembox}[l]{$\Sigma^0_n\LPP$}
  \begin{description}
    \item[ \sf Input] An ill-founded tree $T \subseteq \omega^{<\omega}$ and a path $f$ of $T$.
    \item[ \sf Output] A path $g$ of $T$ such that $g$ is leftmost in $\Sigma^0_{n}(T \oplus f \oplus g)$.
  \end{description}
\end{itembox}

    \begin{itembox}[l]{$\Delta^0_n\LPP$}
  \begin{description}
    \item[ \sf Input] An ill-founded tree $T \subseteq \omega^{<\omega}$ and a path $f$ of $T$.
    \item[ \sf Output] A path $g$ of $T$ such that $g$ is leftmost in $\Delta^0_{n}(T \oplus f \oplus g)$.
  \end{description}
\end{itembox}

    \begin{itembox}[l]{$\Delta^0_n\overline{\LPP}$}
  \begin{description}
    \item[ \sf Input] An ill-founded tree $T \subseteq \omega^{<\omega}$ and a path $f$ of $T$.
    \item[ \sf Output] A path $g$ of $T$ such that $g$ is leftmost in $\Delta^0_{n}(g)$.
  \end{description}
\end{itembox}
\end{definition}

  \begin{lemma}
    $\Sigma^0_n\LPP\equiv_{\W} \Delta^0_n\LPP \equiv_{\W} \Delta^0_n\overline{\LPP}$.
  \end{lemma}
  \begin{proof}
  As a path is $\Sigma^0_n$-definable if and only if it is $\Delta^0_n$-definable, $\Sigma^0_n\LPP$ and $\Delta^0_n\LPP$ are the same.
  For the equivalence of $\Delta^0_n\LPP$ and $\Delta^0_n\overline{\LPP}$, see \cite[Lemma 6.5]{suzuki_yokoyama_fp}
    \end{proof}

    As the same proof of Lemma \ref{Lem LPP implies betaRef}, the following problem is Weihrauch reducible to $\Sigma^0_n\LPP$ for $n > 1$.
     \begin{itembox}[l]{$\Ef\Delta^0_n\beta\RFN$}
  \begin{description}
    \item[ \sf Input] Any set $X$.
    \item[ \sf Output] A $\Delta^0_n\beta$-model $\M$ and computable functions $f_{\TJ},f_{\oplus},f_{\Sigma^1_1}$ such that 
    \begin{align*}
      X \in \M, && (f_{\TJ},f_{\oplus}) \forces_{\M} \ACAo, && f_{\Sigma^1_1} \forces_{\M} \exists X \Pi^0_2.
    \end{align*}
  \end{description}
\end{itembox}

    We see that the converse reduction also holds.

\begin{theorem}\label{thm LPP and RFN}
  For $n > 1$, $\Sigma^0_n\LPP \equiv_{\W} \Ef\Delta^0_n\beta\RFN$.
\end{theorem}
\begin{proof}
  It is enough to show that $\Delta^0_n\overline{\LPP} \leq_{\W} \Ef\Delta^0_n\beta\RFN$ holds.
  Let $T$ be a tree and $f$ be its path. Let $\M, f_{\TJ},f_{\oplus},f_{\Sigma^1_1}$ be such that 
    \begin{align*}
      T,f \in \M, && (f_{\TJ},f_{\oplus}) \forces_{\M} \ACAo, && f_{\Sigma^1_1} \forces_{\M} \exists X \Pi^0_2.
    \end{align*}
    Define a subtree $S$ of $T$ as the set of $\sigma$ in $T$ such that $\M$ believes $\sigma$ is extended to a path. That is,  $S = \{\sigma \in T : \M \models \exists g \in [T] (\sigma \prec g)\}$.
    Then, $S$ is $\M$-computable by Lemma \ref{Lem definability}. Moreover, it is easy to see that there is a uniform procedure to make $S$ from $\M,f_{\TJ},f_{\oplus},f_{\Sigma^1_1}$.
    We now have 
    \begin{itemize}
      \item $S$ is ill-founded because $f \in [S]$,
      \item $S$ is pruned because of the definition, and 
      \item $S = \{\sigma \in T : \exists g \in [T] \cap \Delta^0_n(\M) (\sigma \prec g)\}$ because $\M$ is $\Delta^0_n\beta$-model.
    \end{itemize}
  Let $g \leq_{\T} S$ be the leftmost path of $S$.
  Then, $g$ is leftmost in $\Delta^0_n(\M)$. Since $g$ is $\M$-computable, $g$ is leftmost in $\Delta^0_n(g)$.
  This completes the proof.
\end{proof}

    We then show the separation of $\Sigma^0_{n+1}\LPP$ and $\Sigma^0_n\LPP$.
    For the separation, we use the $\omega$-model reflection of a Weihrauch problem introduced in \cite{suzuki_yokoyama_fp}.  
    \begin{definition}
      Let $\P$ be a Weihrauch problem.
      We define the $\omega$-model reflection of $\P$ as the following problem.
           \begin{screen}
  \begin{description}
    \item[ \sf Input] Any set $X$.
    \item[ \sf Output] A sequence $\M$  and functions $f_{\TJ},f_{\oplus}$ such that 
    \begin{align*}
      &X \in \M \,\&\, (f_{\TJ},f_{\oplus}) \forces_{\M} \ACAo \,\& \\
      &\forall Y \in \M \cap \dom(\P) \exists Z \in \M (Z \in \P(X))
    \end{align*}
  \end{description}
        \end{screen}
      
    \end{definition}

    We note that if there are arithmetical formulas $\theta,\eta$ such that 
    $\dom(\P) = \{X: \theta(X)\}$ and $\P(X) = \{Y : \eta(X,Y)\}$, then 
    the condition $\forall Y \in \M \cap \dom(\P) \exists Z \in \M (Z \in \P(X))$ is equivalent to 
    $\M \models \P$.
    It is also proved that in that case,
    the $\omega$-model reflection of the problem is strictly stronger than the problem. For the details, see Section 4 of \cite{suzuki_yokoyama_fp}. We note that the input and output conditions of $\Sigma^0_n\LPP$ can be written by arithmetical formulas.

  \begin{lemma}[$\ACAo$]\label{Lem model of LPP}
    Let $\M$ be a $\Delta^0_{n+1}\beta$-model and $f_{\TJ},f_{\oplus},f_{\Sigma^1_1}$ be computable functions such that 
    $(f_{\TJ},f_{\oplus}) \forces_{\M} \ACAo$ and $f_{\Sigma^1_1} \forces_{\M} \exists \Pi^0_2$.
    Then, $\M$ is a model of $\Sigma^0_n\LPP$. 
  \end{lemma}
  \begin{proof}
    Let $T \in \M$ be a tree such that $\M \models [T] \neq \varnothing$.
    Then, as in the proof of Theorem \ref{thm LPP and RFN}, there is a $\M$-computable function $g$ such that 
    $\forall f \leq_{\T} \TJ^n(g) (f \in [T] \to g \leq_{l} f)$. 
    In particular, we have $\forall f \leq_{\T} \TJ^{n-1}(g) (f \in [T] \to g \leq_{l} f)$. 
    Thus, we have 
    \begin{align*}
    (\ast) \ \ 
      \exists (g_0 \oplus g_1 \oplus \cdots \oplus g_n) \leq_{\T} \TJ^{n}(\M)
      \Bigl[&\bigwedge_{i < n} g_i = \TJ(g_{i-1}) \\ \land 
      & \forall f \leq_{\T} g_{n-1} (f \in [T] \to g_0 \leq_{\T} f \Bigr].
    \end{align*}
    by taking $g_i = \TJ^i(g)$.
    We note that the condition in the square bracket is $\Pi^0_2$ for $(g_0 \oplus g_1 \oplus \cdots \oplus g_n)$ because $g_n = \TJ(g_{n-1})$. Therefore, the above $(\ast)$ is of the form 
    $\exists X \leq_{\T} \TJ^n(\M) \theta(X)$ for a $\Pi^0_2$ formula $\theta$.
    Since $\M$ is a $\Delta^0_{n+1}\beta$-model, we have 
    \begin{align*}
      \M \models \exists g_0 \forall f \leq_{\T} \TJ^{n-1}(g_0)  (f \in [T] \to g_0 \leq_{\T} f).
    \end{align*}
    Thus, $\M$ believes that $T$ has a $\Delta^0_n$ leftmost path.
    This completes the proof.
  \end{proof}

  \begin{theorem}
      For $n > 0$, $\Sigma^0_n\LPP <_{\W} \Sigma^0_{n+1}\LPP$.
  \end{theorem}
  \begin{proof}
    By Theorem \ref{thm LPP and RFN} and Lemma \ref{Lem model of LPP}, the $\omega$-model reflection of $\Sigma^0_n\LPP$ is Weihrauch reducible to $\Sigma^0_{n+1}\LPP$.
  \end{proof}

\bibliographystyle{plain}
\bibliography{references}

\end{document}